\documentclass[10pt]{amsart}
\usepackage {amssymb,latexsym}
\usepackage [all]{xy}

\newtheorem{theorem}{Theorem}[section]
\newtheorem{lemma}[theorem]{Lemma}
\newtheorem {proposition}[theorem]{Proposition}

\theoremstyle{definition}
\newtheorem{definition}[theorem]{Definition}
\newtheorem{definition and notation}[theorem]{Definition and Notation}

\theoremstyle{remark}
\newtheorem{remark}[theorem]{Remark}
\newtheorem{remarks}[theorem]{Remarks}
\numberwithin {equation}{section}
\def\ra{{\rightarrow}}

\def\ece{{E^{\times}/E^{\times p}}}
\def\ep{{E^{\times p}}}

\def\t{{\times}}
\def\s{{\sigma}}
\def\p{{$p$}}

\def\minac{{Min\'{a}\v{c}}}

\def\fpg{{$\mathbb{F}_p[G]$}}

\def\zp{{\mathbb{Z}/p\mathbb{Z}}}
\def\fpgg{{\mathbb{F}_p[G]}}

\def\fpg{{$\mathbb{F}_p[G]$}}

\def\o{{\oplus}}

\def\fp{{\mathbb{F}_p}}

\def\sb{{\widetilde{\sigma}}}

\def\ecp{{E^{\t p}}}

\def\eb{{\overline{E}}}

\begin{document}
\title{A relative version of Kummer theory}

\author{Vahid Shirbisheh}
\address{Department of Mathematics, Tarbiat Modares University, Tehran, Iran}
\email{shirbisheh@modares.ac.ir}

\date{\today}

\begin{abstract}
Let $E/F$ be a cyclic Galois extension of degree $p^l$ with Galois
group $G$. It is shown that  the Galois module structure of both
sides of the Kummer pairing (for Kummer extensions of $E$) are the
same. In other words, we show that the Kummer duality holds in the
level of finitely generated $G$-modules.
\end{abstract}
\maketitle

\section{Introduction}
\label{sec:intro}

Let $E/F$ be a cyclic Galois extension of degree $p^l$ for a prime
number $p$ and a positive integer $l$. We denote the Galois group
$Gal(E/F)$ by $G$. We also assume $F$ contains a primitive \p th
root of unity. Let $E^\t$ denote the multiplicative group of $E$.
For a positive integer $n$, we set $E^{\t n}=\{ a^n ; a\in E^\t \}$.
The Galois action of $G$ on the abelian group $\ece$ gives it an
\fpg-module structure. The aim of this paper is to establish an
\fpg-module isomorphism between every finitely generated
\fpg-submodule of $\ece$ and the Galois group of its associated
Kummer extension over $E$ as an \fpg-module. Motivation of this
isomorphism comes from certain type of Galois embedding problems.
Besides explaining all details, the other advantage of this
exposition is that the Galois action of $G$ on the above Galois
group is studied explicitly. For applications of this result, we
refer the reader to \cite{mss} and \cite{s3}.

\section{A relative version of Kummer theory}
\label{sec:relative} We assume the Galois group $G=Gal(E/F)$ is
generated by $\s$. We also use notations and definitions of module
theory over \fpg, developed in \cite{s2}. The subgroup of all $n$th
roots of unity in $E^\t$ is denoted by $\mu_n$. If $B$ be a subset
of $E^\t$, then $E_B:=E(B^{1/n})$ is the field extension of $E$
obtained by adding all $n$th roots of elements of $B$ to $E$. Now,
Kummer theory can be summarized in the following theorem:
\begin{theorem} \label{thm:ktheory} Let $n$ be a positive integer and
let $E$ be a field whose characteristic does not divide $n$ which
contains a primitive $n$th root of unity. Let $B$ be a subgroup of
$E^{\t}$ containing $E^{\t n}$. Then $E_B$ is Galois over $E$ and
the Galois group $N_B=Gal(E_B/E)$ is abelian of exponent $n$.
Moreover, there is a $\mathbb{Z}/n\mathbb{Z}$-bilinear map
\begin{equation}
\label{f:kpairing1} N_B \t B/E^{\t n}\ra \mu_n
\end{equation}
\[
(\tau,[x])\mapsto
\langle\tau,[x]\rangle:=\tau(\sqrt[n]{x})/\sqrt[n]{x}, \quad \tau\in
N_B, x\in B.\]
The extension $E_B/E$ is finite if and only if
$B/E^{\t n}$ is finite, in this case $B/E^{\t n}$ is isomorphic to
$N_B$. Furthermore, the map $B\mapsto E_B$ is a bijection between
the set of subgroups of $E^\t$ containing $E^{\t n}$ and Galois
extensions of $E$ whose Galois groups over $E$ are abelian of
exponent $n$.
\end{theorem}
A field extension of $E$ of the form $E_B$ described in the above
theorem is called a {\it Kummer extension of $E$ of exponent $n$}.
The pairing \ref{f:kpairing1} is called {\it the Kummer pairing}.
From now on, we let $n$ be a prime number \p. Then, the Kummer
pairing is a non-degenerate $\fp$-bilinear map. The normal closure
of a Kummer extension $E_B$ of $E$ over $F$ is denoted by $\eb_B$.
It is the Kummer extension associated with $\overline{B}$, the
\fpg-submodule of $\ece$ generated by elements of $B$, see \cite{w}.
To describe the action of $G=Gal(E/F)$ on $Gal(\eb_B/E)$, or
equivalently the \fpg-module structure of $Gal(\eb_B/E)$, the
following remark is useful.
\begin{remark}
\label{r:3actions} There are three Galois actions of $G$ on abelian
groups occurring in the Kummer pairing \ref{f:kpairing1}. Since F
contains a primitive \p th root of unity, the action of $G$ on
$\mu_p$ is trivial. It is clear that $E^{\t p}$ is invariant under
the Galois action of $G$. Therefore, if $B$ is a subgroup of $E^\t$
containing $E^{\t p}$ and invariant under the action of $G$, then
$B/E^{\t p}$ is endowed with an action of $G$ defined by
$\s([a]):=[\s(a)]$ for $a\in B$. It also follows from Theorem 1 of
\cite{w} that $E_B=\eb_B$, and consequently $E_B$ is Galois over
$F$. In this case, let $N_B$ (resp. $H_B$) be the Galois group of
$E_B$ over $E$ (resp. $F$), as it is illustrated in the following
diagram:
\begin{equation}
\xymatrix{ E_B \ar@{-}[d] \ar@/^2pc/@{-}[dd]^{H_B=Gal(E_B/F)} \ar@/_2pc/@{-}[d]_{N_B=Gal(E_B/E)}\\
E \ar@/_2pc/@{-}[d]_{G=Gal(E/F)} \ar@{-}[d]\\
F }
\end{equation}
Regarding the extension $1\ra N_B \ra H_B \ra G \ra 1$, let $\sb$ be
a lift of $\s$ in $H_B$. The action of $G$ on $H_B$ is defined by
conjugation, i.e. $\s(h):=\sb h \sb^{-1}$ for all $h\in H_B$. Since
$N_B$ is normal in $H_B$, this action induces an action of $G$ on
$N_B$.  Let $\s_0$ be another lift of $\s$ in $H_B$. Then $\s_0=\sb
n_0$ for some $n_0\in N_B$ and due to the fact that $N_B$ is
abelian, for $n\in N_B$, we have $ \s_0 n \s_0^{-1}=\sb n_0 n
n_0^{-1} \sb^{-1}=\sb n\sb^{-1}$. This shows that the above action
of $G$ on $N_B$ does not depend on the choice of the lift of $\s$.
Clearly, it does not depend on the choice of the generator of $G$
neither.
\end{remark}

\begin{proposition} (\cite{w}) The Kummer pairing preserves the Galois module structure of the involving abelian groups. More precisely, for
all $b\in B$ and $\tau\in N_B$, we have
\begin{equation}
\label{f:kpres} \langle\s(\tau), \s([b])\rangle=\s(\langle\tau,
[b]\rangle)= \langle\tau, [b]\rangle.
\end{equation}
\end{proposition}
\begin{proof} We keep using notations and assumptions of Remark \ref{r:3actions}.
Let $b\in B$ and let $y=\sqrt[p]{b}$ be a \p th root of $b$ in
$E_B$. Then, we have $ \langle\s(\tau), \s([b])\rangle =
\langle\s(\tau), [\s(b)]\rangle= \frac{\sb \tau \sb^{-1}
\sqrt[p]{\s(b)}}{\sqrt[p]{\s(b)}}= \frac{\sb \tau \sb^{-1}
\sqrt[p]{\s(y^p)}}{\sqrt[p]{\s(y^p)}}= \frac{\sb \tau \sb^{-1}
\sqrt[p]{(\sb(y))^p}}{\sqrt[p]{(\sb(y))^p}}= \frac{\sb \tau \sb^{-1}
\sb(y)}{\sb(y)}= \frac{\sb \tau (y)}{\sb(y)} = \sb \left(\frac{\tau
(\sqrt[p]{b})}{(\sqrt[p]{b})} \right)= \sb (\langle\tau,
[b]\rangle)= \s(\langle\tau, [b]\rangle)=\langle\tau, [b]\rangle$.
The last equality follows from the fact that the action of $G$ on
$\mu_p$ is trivial.
\end{proof}
Our next goal is to show that if $B/\ep$ is a finitely generated
\fpg-submodule of $\ece$, then the \fpg-module structures of $B/\ep$
and $N_B$ are the same. Here, our approach is based on the structure
of $N_B$ as a Galois group with a $G$-action. Let $a\in E^\t$. In
the following remarks, we make some observations about the group
structure and the \fpg-module structure of
$Gal(\overline{E(a^{1/p})}/E)$.
\begin{remarks}
\label{r:gaction}
Let $s$ be the smallest natural number such that
$x^s a$ is a \p th power in $E$.
\begin{itemize}
\item [(i)] It is clear from Theorem 1 of \cite{w} that $\overline{E(a^{1/p})}$
is a splitting field of the following family of polynomials in
$E[y]$
\begin{equation*}
\{   P_0(y)=y^p-a, P_1(y)=y^p-xa,\cdots, P_{s-1}(y)=y^p-x^{s-1}a\},
\end{equation*}
where $x=\s-1$ as in \cite{s2}. Let $\zeta$ be a primitive \p th
root of unity, and let $\alpha_0$ be a root of $P_0(y)$. Then, for
$i=1,\cdots,s-1$, we define $\alpha_i:=x^i \alpha_0$. Now, it is
easy to see that for $i=0,\cdots,s-1$, the roots of $P_i(y)$ are
$\alpha_i,\zeta\alpha_i,\cdots,\zeta^{p-1}\alpha_i$. Furthermore, we
note that an arbitrary element $(n_0,\cdots,n_{s-1})$ in
$(\zp)^s\simeq Gal(\overline{E(a^{1/p})}/E)$ permutes these roots by
the formula
$(n_0,\cdots,n_{s-1})(\zeta^j\alpha_i)=\zeta^{j+n_i}\alpha_i$.
\item [(ii)] Since $x^sa$ is a \p th power in $E$, we have $(x\alpha_{s-1})^p=x(\alpha_{s-1})^p=x(x^{s-1}a)=x^sa=b^p$
for some $b\in E$. Thus, $x\alpha_{s-1}$ is a \p th root of $b^p$.
On the other hand $b$ is a \p th root of $b^p$ too. So,
$x\alpha_{s-1}=\zeta^k b$ for an integer $k$ and consequently
$x\alpha_{s-1}\in E$. Therefore, $x^j\alpha_i \in E$, if $i+j \geq
s$. On the other hand, by our definition, we have
$x^j\alpha_i=\alpha_{i+j}$, whenever $i+j<s$.
\end{itemize}
\end{remarks}
We know from \cite{s2} that \fpg, as a ring, is isomorphic to
$\fp[x]/x^q$, where $q=p^l$.
\begin{lemma}
\label{l:rho} With the above notation, let $\rho:A\ra
Gal(\overline{E(a^{1/p})}/E)$ as an \fpg-homomorphism be defined by
$\rho(1)=(0,\cdots,0,1)$. Then, the kernel of $\rho$ is $\langle
x^s\rangle$.
\end{lemma}
\begin{proof}
By Proposition 1.2 of \cite{s2}, the kernel of $\rho$ is of the form
$\langle x^m\rangle$ for some $m\leq q$. If $m>s$, then we obtain an
injective map from $A/\langle x^m\rangle$ into
$Gal(\overline{E(a^{1/p})}/E)\simeq (\zp)^s$, which is impossible
because it implies that $p^m$, the cardinality of $A/\langle
x^m\rangle$, is less than or equal $p^s$, the cardinality of
$(\zp)^s$. Now, by showing that $\rho(x^{s-1})\neq id$, we prove
that $m$ cannot be less than $s$. For $0\leq n \leq s-2$, let
$\rho(x^{n})=(n_0,\cdots,n_{s-1})$ and let
$\rho(x^{n+1})=(m_0,\cdots,m_{s-1})$.\\
{\bf Claim:} Let $n_i$ be the last non-zero component of
$\rho(x^{n})$. If $i>0$, then $m_{i-1}\neq 0$.\\
To see this, we have to compute $\rho(x^{n+1})(\alpha_{i-1})$. With
the notations of Remark \ref{r:3actions}, we have $
\rho(x^{n+1})\rho(x^n)=\rho(x^{n+1}+x^{n}) =
\rho((x+1)(x^{n}))=(x+1)\rho(x^{n})=\sb(\rho(x^{n})))= \sb
 \rho(x^{n})\sb^{-1}=(x+1)\rho(x^{n})(x+1)^{-1} $. Thus, we have
 \begin{equation}
\rho(x^{n+1})=(x+1)\rho(x^{n})(x+1)^{-1} \rho(x^{n})^{-1}.
 \end{equation}
Hence, $\rho(x^{n+1})(\alpha_{i-1})= (x+1)\rho(x^{n})(x+1)^{-1}
\rho(x^{n})^{-1}(\alpha_{i-1})=(x+1)\rho(x^{n})(x+1)^{-1}(\zeta^{-n_{i-1}}\alpha_{i-1})=
\zeta^{-n_{i-1}}(x+1)\rho(x^{n})(1-x+x^2-\cdots)(\alpha_{i-1})=
\zeta^{-n_{i-1}}(x+1)\rho(x^{n})\left(\frac{\alpha_{i-1}(x^2\alpha_{i-1})\cdots}{(x\alpha_{i-1})(x^3\alpha_{i-1})\cdots}\right)
=\zeta^{-n_{i-1}}(x+1)\rho(x^{n})\left(\frac{\alpha_{i-1}(x^2\alpha_{i-1})\cdots}{\alpha_{i}(x^3\alpha_{i-1})\cdots}\right)=\star$.
We note that for $r>1$, $x^r\alpha_{i-1}$ is either $\alpha_{i+r-1}$
or an element of $E$, see Remarks \ref{r:gaction}. In both cases,
$\rho(x^n)$ does not change it. Thus, we have $ \star=
\zeta^{-n_{i-1}}(x+1)\left(\frac{\zeta^{n_{i-1}}\alpha_{i-1}(x^2\alpha_{i-1})\cdots}{\zeta^{n_{i}}\alpha_{i}(x^3\alpha_{i-1})\cdots}\right)
=\zeta^{-n_{i}}(x+1)\left(\frac{\alpha_{i-1}(x^2\alpha_{i-1})\cdots}{\alpha_{i}(x^3\alpha_{i-1})\cdots}\right)
=\zeta^{-n_{i}}(x+1)\left(\frac{\alpha_{i-1}(x^2\alpha_{i-1})\cdots}{(x\alpha_{i-1})(x^3\alpha_{i-1})\cdots}\right)
=\zeta^{-n_{i}}(x+1)((x+1)^{-1}(\alpha_{i-1}))=\zeta^{-n_{i}}\alpha_{i-1}$.
Therefore, $m_{i-1}=-n_{i}\neq 0$.\\
Now, since the component with index $s-1$ in $\rho(1)$ is non-zero,
the component with index $s-2$ in $\rho(x)$ is non-zero and one can
argue by induction that $\rho(x^{k})$ has a non-zero component for
$0\leq k< s$. Thus, $\rho(x^k)\neq id$ for $0\leq k < s$.
\end{proof}
We call the following theorem the {\it relative Kummer theory}.
\begin{theorem}
\label{thm:relativekt} Let $E/F$ be a cyclic extension of degree
$q=p^l$ with Galois group $G$ and let $F$ contain a primitive \p th
root of unity. Let $B$ be a subgroup of $E^\t$ containing $E^{\t p}$
and invariant under the Galois action of $G$. If $B/E^{\t p}$ is a
finitely generated \fpg-module, it has the same \fpg-module
structure as $N_B=Gal(E_B/E)$.
\end{theorem}
\begin{proof} We prove the statement by induction on the number of summands in the
decomposition of $B/E^{\t p}$ into direct sum of cyclic modules. In
the case $B/E^{\t p}=0$ there is nothing to prove. So, let
$B/\ecp\simeq B_1/\ecp\o \cdots \o B_n/\ecp$ be a decomposition of
$B/\ecp$ into direct sum of cyclic modules, and let the statement of
the theorem be true for \fpg-submodules of $\ece$ whose
decompositions have less summands than $n$.
\par
For $i=1,\cdots,n$, let $\dim_{\fp}B_i/\ecp=l_i$. We choose $b_i\in
B_i$ such that $[b_i]$ is a generator of $B_i/\ecp$ as an
\fpg-module. Consider $$B[1]=B_2/\ecp\o \cdots \o B_n/\ecp $$ as a
submodule of $B/\ecp$, and let $N[1]=Gal(E_B/E_{B[1]})$, as
illustrated in the following diagram:
\begin{equation}
\xymatrix{E_B \ar@{-}[d] \ar@/_2pc/@{-}[dd]_{N_B} \ar@/^2pc/@{-}[d]^{N[1]}\\
E_{B[1]}\ar@{-}[d] \ar@/^2pc/@{-}[d]^{N_{B[1]}}\\
E}
\end{equation}
To prove that $N_B$ and $B/\ecp$ have the same \fpg-module
structure, we first show that $N[1]\simeq B_1/\ecp$. Let
\[
C=\{x^{j_i}b_i; i=1,\cdots,n \; \text{and}\; 0\leq j_i \leq l_i-1
\}.
\]
By Theorem 1 of \cite{w}, $E_B=E(C^{1/p})$. Therefore,
\[
E_B=E_{B[1]}(b_1^{1/p},(xb_1)^{1/p},\cdots,(x^{l_1-1}b_1)^{1/p}).
\]
Let $a=b_1$ and $s=l_1$. Then, Remarks \ref{r:gaction} applies to
the extension $E_B=E_{B[1]}(a^{1/p})$ over $E_{B[1]}$ and similar to
Lemma \ref{l:rho}, we define an \fpg-homomorphism $\rho:A\ra N[1]$
whose kernel is $\langle x^s\rangle$. This map gives rise to an
injective \fpg-homomorphism, say $\tilde{\rho}$ of $A/\langle
x^s\rangle\simeq B_1/\ecp$ into $N[1]=Gal(E_B/E_{B[1]})=\zp^s$.
Since the source and the target of $\tilde{\rho}$ have the same
cardinality, it is an \fpg-isomorphism. Therefore,
\begin{equation}
N[1]\simeq B_1/E^{\t p}.
\end{equation}
Now, we claim that $N[1]$ is an \fpg-submodule of $N_B$. To see
this, let $\tau\in N[1]$. Since $B[1]$ is an \fpg-submodule of
$B/E^{\t p}$, for every $b\in B[1]$, $\s^{-1}b\in B[1]$. Thus, by
Proposition 2.3, for every $b\in B[1]$ we have
$\langle\s\tau,b\rangle=\langle\s\tau,\s\s^{-1}b\rangle=\langle\tau,\s^{-1}b\rangle=1$.
This implies
that $\s\tau\in N[1]$. \\
On the other hand, by Kummer theory, we have the following
decomposition of $N_B$ into a direct sum of two abelian subgroups:
\begin{equation}
N_B\simeq N[1]\o N_{B[1]}.
\end{equation}
We claim that $N_{B[1]}$ is an \fpg-submodule of $N_B$. To see this,
consider an element of the basis of $E_B$ over $E_{B[1]}$, say
$(x^jb_1)^{1/p}$. Then, for $\tau\in N_{B[1]}$, we have $
\s(\tau)(x^jb_1)^{1/p}= \s\tau\s^{-1} (x^jb_1)^{1/p}= (x+1) \tau
(x+1)^{-1} (x^jb_1)^{1/p} = (x+1) \tau (1-x+x^2-\cdots)
(x^jb_1)^{1/p}= \star$. Now, one notes that every term of
$(1-x+x^2-\cdots) (x^jb_1)^{1/p}$ is again an element of the basis
of $E_B$ over $E_{B[1]}$, so $\tau$ acts trivially on it and we have
$ \star= (x+1)(1-x+x^2-\cdots) (x^jb_1)^{1/p}=(x+1)(x+1)^{-1}
(x^jb_1)^{1/p}=(x^jb_1)^{1/p}$. This shows that $\s\tau\in
N_{B[1]}$. Therefore, 2.7 can be considered as a decomposition of
$N_B$ into a direct sum of two \fpg-submodules. By the induction
hypothesis, $N_{B[1]}\simeq B[1]$. Therefore, using 2.6 and 2.7 we
have $N_B\simeq B_1/E^{\t p}\o B[1]=B/E^{\t p}$ as two \fpg-modules.
\end{proof}

\section{the second proof}
\label{sec:second} In this section, we explain another approach to
prove the relative Kummer theory. This approach has been suggested
to the author by \minac, and was appeared concisely in \cite{mss}.
Here, we first use the fact that the Kummer pairing preserves the
$G$-action, Proposition 2.3. Then, we prove an analogue of Lemma
\ref{l:rho}  to show cyclic \fpg-modules are self-dual. The
advantage of this approach is its simplicity, also it is shorter.
However, one does not see the \fpg-module structure of $N_B$
concretely here.
\par
It is well-known that the Kummer pairing defines a group isomorphism
between $N_B$ and $Hom_{\mathbb{Z}}(B/\ep,\fp)$ as follows:
\begin{eqnarray}
\label{f:psi} \quad \quad \Psi:N_B \ra Hom_{\mathbb{Z}}(B/\ep,\fp)
\end{eqnarray}
\begin{eqnarray*}
\qquad \qquad\quad \tau\mapsto \Psi_{\tau} :&B/\ep&\ra \fp \\
&[b]&\mapsto \langle\tau,[b]\rangle
\end{eqnarray*}
\par
 The following definition describes the \fpg-module structure of
 the abelian group $Hom_{\mathbb{Z}}(B/\ep,\fp)$, the interested reader can find more
 details about this in \cite{b}, page 48.
\begin{definition}
Let a group $G$ act on two abelian groups $H$ and $K$. Then, $G$
acts on $Hom_{\mathbb{Z}}(H,K)$ as follows:
\[
g\theta (h)=g\theta (g^{-1}h),
\]
where $g\in G$, $\theta\in Hom_{\mathbb{Z}}(H,K)$ and $h\in H$.
\end{definition}
One notes that the above \fpg-module structure on
$Hom_{\mathbb{Z}}(H,K)$ is different from its usual module structure
as the Hom functor of two \fpg-modules. We consider only this
\fpg-module structure in this paper.
\begin{remark}
\label{r:psii} First, we note that since $B/\ep$ is of exponent \p,
every homomorphism in $Hom_{\mathbb{Z}}(B/\ep,\fp)$ can be
considered mod \p. So, this Hom is the same as
$Hom_{\fp}(B/\ep,\fp)$, which is the dual of $B/\ep$ as a vector
space over $\fp$.  Now, by considering the above \fpg-module
structure on $Hom_{\fp}(B/\ep,\fp)$, it is easy to see that $\Psi$
as defined in \ref{f:psi} is an \fpg-isomorphism. To see this, let
$\s\in G$ and $\tau\in N_B$, then we have $
[\Psi(\s\tau)]([b])=\langle
\s\tau,[b]\rangle=\langle\tau,\s^{-1}[b]\rangle=\langle
\s\tau,\s\s^{-1}[b]\rangle=\langle\tau,\s^{-1}[b]\rangle=[\Psi(\tau)](\s^{-1}[b])=\s([\Psi(\tau)](\s^{-1}[b]))$.
Since $F$ contains a \p th root of unity, we have
$[\Psi(\s\tau)]([b])=[\s(\Psi(\tau))]([b])$ for every $b\in B$.
\end{remark}
\begin{proposition} Let $M$ be a finitely generated \fpg-module. Then $M$ and $Hom_{\fp}(M,\fp)$ are
isomorphic as \fpg-modules.
\end{proposition}
\begin{proof} If $M\simeq M_1\o M_2$, it is easy to see that, for $i=1,2$, $Hom_{\fp}(M_i,\fp)$ can be considered as an \fpg-submodule of
$Hom_{\fp}(M,\fp)$. Thus $Hom_{\fp}(M,\fp) \simeq
Hom_{\fp}(M_1,\fp)\o Hom_{\fp}(M_2,\fp)$ as \fpg-modules.
 Therefore it is enough to prove the statement in the case that $M$ is
 cyclic. Let $\dim_\fp M=l$.
 Regarding the automorphism of \fpg\, induced by the map $\s\mapsto\s^{-1}$, it is clear that the ideal generated by $(\s^{-1}-1)^l$ is isomorphic
 with the ideal generated with $(\s-1)^l$. Thus, one can consider $M$ as $\fpgg/\langle y^l\rangle$, where $y=\s^{-1}-1$. We define $f\in
 Hom_{\fp}(M,\fp)$ by extending rules
 $f(1)=f(x)=\cdots=f(x^{l-2})=0$ and $f(x^{l-1})=1$ linearly to $M$. It follows from the \fpg-module structure of
 $Hom_{\fp}(M,\fp)$ that $((\s-1)f)(m)=f((\s^{-1}-1)m)$ for $m\in
 M$. By repeating this we get $((\s-1)^{l-1}f)=f(y^{l-1})=1.$ Hence,
 we obtain $ \dim_\fp(\langle f\rangle)\geq l$.
Since $Hom_{\fp}(M_1,\fp)$ has dimension $l$ over $\fp$, we conclude
that $Hom_{\fp}(M_1,\fp)=\langle f\rangle.$ Thus, both $M$ and
$Hom_{\fp}(M_1,\fp)$ are cyclic \fpg-modules of dimension $l$, and
so they are isomorphic.
\end{proof}
It is clear that Theorem \ref{thm:relativekt} follows from Remark
\ref{r:psii} and the above proposition.\\

\noindent{\bf Acknowledgment.} I would like to thank Prof. J\'{a}n
\minac\,, Prof. Ajneet Dhillon, and Prof. John Swallow for many
useful discussion and comments, in particular, Prof. J\'{a}n
\minac\, who guided me to the second proof.
\bibliographystyle {amsalpha}

\end{document}